\documentclass[letterpaper, 10 pt, conference]{ieeeconf}
\IEEEoverridecommandlockouts
\usepackage{amsmath,amssymb,amsfonts,mathtools}
\usepackage{float}
\usepackage[dvipsnames]{xcolor}
\usepackage{tikz}
    \usetikzlibrary{calc}
    \usetikzlibrary{positioning}
    \usetikzlibrary{arrows}

\title{\LARGE \bf
Observer-Based Performance-Barrier Event-Triggered Control of $2\times2$ Linear Hyperbolic PDEs
}
\author{Eranda Somathilake$^{1}$ and Mamadou Diagne$^{1}$
\thanks{*This work was supported by the NSF CAREER Award CMMI-2302030 and  the NSF grant CMMI-2222250.}
\thanks{$^{1}$Eranda Somathilake and Mamadou Diagne are with the Department of Mechanical and Aerospace Engineering,
        University of California San Diego, CA 92093, USA
        {\tt\small esomathilake@ucsd.edu, mdiagne@ucsd.edu}}%
}

\newtheorem{theorem}{Theorem}
\newtheorem{lemma}{Lemma}
\newtheorem{proposition}{Proposition}

\newtheorem{definition}{Definition}
\newtheorem{assumption}{Assumption}

\begin{document}
\maketitle
\thispagestyle{empty}
\pagestyle{empty}
\begin{abstract}
Performance-barrier event-triggered control (P-ETC) is a methodology implemented to increase the dwell-times between events while still preserving a prescribed performance of the system under event-triggered control (ETC). This is achieved by considering the performance residual of the system, which is a measure of the system performance with respect to the prescribed performance. This allows the Lyapunov function candidate to deviate from decreasing monotonically. In order to determine the performance residual, it is required to know the full-state information, leading to all work related to P-ETC to be under full-state feedback. In this article, we propose a novel dynamic performance-barrier under output feedback with an exponentially convergent observer. We consider event-triggered boundary control of a class of $2\times2$ linear hyperbolic PDEs with anti-collocated measurements with the control input. Under the proposed P-ETC mechanism, we prove the existence of a minimum dwell-time, and show the global exponential stability of the spatial $L^2$ norm of the solution of the system. Simulation results are presented to validate the theoretical claims.
\end{abstract}
\section{Introduction}

Event-triggered control (ETC) is a feedback control strategy in which the control input is updated aperiodically at discrete instants, referred to as \textit{events}, rather than at fixed sampling times \cite{tabuada2007event,heemels2012introduction}. The control input is applied in a zero-order hold (ZOH) fashion between events; that is, the control input remains constant until the next triggering instant. This framework enables the implementation of feedback control on digital platforms while reducing unnecessary updates. Increasing inter-event times (dwell-times) yields practical benefits. These include communication bandwidth requirements, lower power consumption in networked or remote components, and decreased actuator wear due to fewer updates. Consequently, ETC is particularly attractive for large-scale and networked control systems. 

The events under ETC are determined by monitoring if triggering function defined by the measured states, and user-defined tuning parameters satisfy a particular condition. The triggering condition under ETC can be dynamic or static. In static ETC, the triggering rule depends directly on the system states whereas in dynamic ETC, the triggering rule is based on an auxiliary dynamic variable whose evolution depends on the system's states, providing additional design flexibility and often longer dwell-times \cite{girard2015dynamic}. 

In this article, we consider a class of $2\times2$ linear hyperbolic partial differential equations (PDEs) with anti-collocated boundary actuation and  boundary measurements under a dynamic ETC framework. Hyperbolic PDEs arise in the modeling of transport processes including open channel fluid flow\cite{coron2007,litrico2009,diagne2017backstepping,somathilake2025_etc,somathilake2024_sediment}, traffic systems \cite{yu2019traffic}, and gas pipeline systems  \cite{bhathiya2025_gas}, among others. In such systems, actuation and sensing are typically restricted to the spatial boundaries, while the state evolves over a distributed domain. As a result, full-state measurements are generally unavailable, and control implementation relies on limited boundary information. Moreover, these systems often span large spatial domains, leading to geographically separated sensors and actuators that communicate over bandwidth-constrained networks. Therefore, output-feedback ETC provides a practically motivated framework for hyperbolic PDEs, as it reduces communication and actuation updates.

ETC mechanisms can be designed in two ways: by emulation and by co-design. The former samples a pre-designed continuous control signal at each event while the latter designs the triggering mechanism and the control input simultaneously. The emulation approach is selected under the framework proposed in this paper. Continuous-time boundary control design of linear hyperbolic PDEs has been an area that has been studied extensively with the backstepping method being one of the most prominent methods used \cite{vazquez2026_backstepping_survery}. Some examples of the PDE backstepping method applied to first order linear hyperbolic PDEs include \cite{krstic2008_1xhyperbolic} for a single PDE and \cite{vazquez2011_2x2hyperbolic} for coupled $2\times2$ linear hyperbolic PDEs with boundary measurements. These results were extended to more general classes of  linear hyperbolic PDEs in \cite{meglio2013,hu2016}. We refer the reader to \cite{anfinsen2019adaptive} for results on adaptive control of linear hyperbolic PDEs using the backstepping method. The continuous-time control design considered here is based on \cite{vazquez2011_2x2hyperbolic} while making necessary adjustments for the anti-collocated actuation and measurement which is the focus here. Dynamic ETC via emulation has been applied to systems of PDEs in works such as \cite{rathnayake2021observer,espitia2020observer,bhathiya2026_etc_stability}. Notably, \cite{bhathiya2026_etc_stability} provides an ETC condition that provides global exponential stability (GES) of the spatial $L^2$ norm of the system states compared to the global exponential convergence (GEC) results in previous work. This has been achieved by incorporating switching dynamics to the dynamic variable in the triggering function.

Under ETC, the event times are determined by a triggering condition designed to preserve desired stability and performance properties. By updating the control input only when required by the observed system's state, ETC can increase dwell-times compared to periodic control where the input is updated periodically \cite{heemels2012introduction}. In periodic control, the sampling frequency must be chosen to accommodate worst-case scenarios, often leading to conservative and unnecessarily frequent updates. In contrast, ETC adapts the update instants to the current system behavior, allowing longer dwell-times whenever the system operates away from critical conditions. 

In this article, we propose an observer-based dynamic \textit{performance-barrier} ETC (P-ETC) methodology that further increases dwell-times, based on the work in \cite{ong2024performance,rathnayake2025performance,zhang2025performance}. In \cite{ong2024performance}, the authors introduced the concept of a performance residual, which quantifies how well the system satisfies a prescribed performance specification and allows for longer dwell-times accordingly. The performance criterion is defined by requiring the Lyapunov function to remain bounded by an exponentially decaying function determined by the initial condition. This decaying bound is referred to as the \textit{performance-barrier}, and the performance residual is defined as the difference between the value of the Lyapunov function and this barrier. The framework was developed for ordinary differential equations (ODEs) in \cite{ong2024performance} and later extended to a class of reaction-diffusion PDEs in \cite{bhathiya2026_etc_stability} and to a class of linear hyperbolic PDEs in \cite{zhang2025performance}. By incorporating the performance residual into a dynamic triggering condition, the requirement for the Lyapunov function to decrease monotonically can be relaxed, thereby allowing longer dwell times. However, evaluating the performance residual requires knowledge of the Lyapunov function at all times. Consequently, existing P-ETC results are restricted to full-state feedback systems. In this paper, we propose a novel dynamically evolving performance-barrier applicable to systems with measurement-only feedback. Hence, we obtain an estimated performance barrier based on the measured states to be used in the triggering function. We incorporate the results in \cite{bhathiya2026_etc_stability} to present an observer-based P-ETC methodology that achieves GES of the spatial $L^2$ norm, as opposed to prior state-feedback P-ETC designs for PDEs in \cite{rathnayake2025performance,zhang2025performance} with GEC of the spatial $L^2$ norm.

The paper is structured as follows. Section~\ref{sec:problem_statement} presents the system under consideration and the preliminaries required for the following sections. In Section~\ref{sec:p_etc}, the P-ETC mechanism is presented along with the stability results under the proposed control mechanism. Finally, numerical simulations and conclusions are presented in Sections~\ref{sec:problem_statement} and \ref{sec:conclusion} respectively.
\subsection{Notations}
Let $\mathbb{R}$, $\mathbb{R}_{>0}$, and $\mathbb{R}_{\geq0}$ be the set of real numbers, the set of real positive numbers, and the set of nonnegative real numbers respectively. Let $\mathbb{N}$ be the set of natural numbers including zero. By $L^2((0,1);\mathbb{R})$ we denote the equivalence class of Lebesgue measurable functions $f:[0,1]\rightarrow\mathbb{R}$ such that $\|f\|_{L^2((0,1);\mathbb{R})}=\left(\int_0^1|f(x)|^2dx\right)^{1/2}<\infty$. Let $\|f\|=\|f\|_{L^2((0,1);\mathbb{R})}$ and $\|f\|_\infty=\|f\|_{L^\infty((0,1);\mathbb{R})}$. $u[t]$ denotes the profile of $u$ at a certain $t\geq0$, \textit{i.e.}, $(u[t])(x)=u(x,t)$, for all $x\in[0,1]$. Define $\mathcal{C}^0(I;L^2((0,1);\mathbb{R}))$ as the space of continuous functions $u[t]$ for an interval $I\subseteq\mathbb{R}_{\geq 0}$ such that $I\ni t\rightarrow u[t]\in L^2((0,1);\mathbb{R})$. 

\section{Preliminaries and Problem statement}\label{sec:problem_statement}
We consider a $2\times2$ linear hyperbolic system in the canonical form. The independent variables $t\geq0$ and $x\in[0,1]$ denote the temporal and spatial domains, respectively. The distributed states are defined as $(u,v)^T$. The system is given by
\begin{align}
    &\partial_t u(x,t) = -\lambda_1(x) \partial_x u(x,t) + c_1(x) v(x,t),\label{sysu}\\
    &\partial_t v(x,t) = \lambda_2(x) \partial_x v(x,t) + c_2(x) u(x,t),\label{sysv}\\
    &u(0,t) =q v(0,t),\quad
    v(1,t) = \rho u(1,t) + U(t),\label{sysbc}
\end{align}
along with the available measurement $y(t)=v(0,t)$. Here, $\lambda_1(x),\lambda_2(x)\in \mathcal{C}^1((0,1);\mathbb{R}_{>0})$ are the transport velocities, $c_1(x),c_2(x)\in \mathcal{C}^0((0,1);\mathbb{R})$ are the spatially varying coupling coefficients, $q\neq0$ and $\rho\neq0$ are boundary parameters, and $U(t)$ is the control input to be determined. The initial conditions are assumed to be such that $(u^0,  v^0)^T \in  L^2((0, 1);\mathbb{R}^2)$. The boundary parameters are assumed to satisfy the following dissipative property.
\begin{assumption}\label{assum:rhoq}
We assume that the boundary parameters are small enough such that
\begin{equation}
    |\rho q|\leq 0.5.
\end{equation}
In addition, define
\begin{align}
    &\phi_1(x)=\int_0^x \tfrac{1}{\lambda_1(\xi)}d\xi,\quad \phi_2(x)=\int_0^x \tfrac{1}{\lambda_2(\xi)}d\xi,\label{phi}
\end{align}
We aim to determine the control input $U(t)$ applied in a zero-order hold fashion between events times. This increasing sequence of event times is denoted by  $I=\{t_j\}_{j\in\mathbb{N}},t_0=0$. The stabilizing control input $U(t)$ is obtained via emulation of a continuous-time stabilizing control signal as defined subsequently. 
\end{assumption}
\subsection{Continuous-Time Output Feedback Control and its Emulation}
Let an observer be designed with the states $(\hat u,\hat v)^T$, with initial conditions$(\hat u^0,  \hat v^0)^T \in  L^2((0, 1);\mathbb{R}^2)$. The observer error defined as  $(\tilde u,\tilde v)^T= (u,v)^T - (\hat u,\hat v)^T$ such that
\begin{align}
    &\partial_t \hat u(x,t) = -\lambda_1(x) \partial_x \hat u(x,t) + c_1(x) \hat v(x,t) + p_1(x)\tilde v(0,t),\label{obs-sysu}\\
    &\partial_t \hat v(x,t) = \lambda_2(x) \partial_x \hat v(x,t) + c_2(x) \hat u(x,t)+ p_2(x)\tilde v(0,t),\label{obs-sysv}\\
    &\hat u(0,t) = q v(0,t),\quad
    \hat v(1,t) = \rho \hat u(1,t) + U(t).\label{obs-sysbc}
\end{align}
Hence the error dynamics satisfy the equations
\begin{align}
    &\partial_t \tilde u(x,t) = -\lambda_1(x) \partial_x \tilde u(x,t) + c_1(x) \tilde v(x,t) - p_1(x)\tilde v(0,t),\label{err-sysu}\\
    &\partial_t \tilde v(x,t) = \lambda_2(x) \partial_x \tilde v(x,t) + c_2(x) \tilde u(x,t) - p_2(x)\tilde v(0,t),\label{err-sysv}\\
    &\tilde u(0,t) = 0,\quad
    \tilde v(1,t) = \rho \tilde u(1,t).\label{err-sysbc}
\end{align}
Consider the backstepping transformation
\begin{align}
    \tilde{u}(x,t) & =\tilde{\alpha}(x,t)-\int_0^x P^{\alpha\alpha}(x, \xi) \tilde{\alpha}(\xi,t) d\xi\nonumber\\&\quad-\int_0^x P^{\alpha\beta}(x, \xi) \tilde{\beta}(\xi,t) d\xi,\label{err-backstepping1}\\
\tilde{v}(x,t) & =\tilde{\beta}(x,t)-\int_0^x P^{\beta\alpha}(x, \xi) \tilde{\alpha}(\xi,t) d\xi\nonumber\\&\quad-\int_0^x P^{\beta\beta}(x, \xi) \tilde{\beta}(\xi,t) d\xi,\label{err-backstepping2}
\end{align}
and its inverse 
\begin{align}
\tilde{\alpha}(x,t)=&\tilde{u}(x,t)+\int_0^xR^{uu}(x,\xi)\tilde{u}(\xi,t)d\xi\nonumber\\&+\int_0^xR^{uv}(x,\xi)\tilde{v}(\xi,t)d\xi,\label{err-backstepping1-inverse}\\
\tilde{\beta}(x,t)=&\tilde{v}(x,t)+\int_0^xR^{vu}(x,\xi)\tilde{u}(\xi,t)d\xi\nonumber\\&+\int_0^xR^{vv}(x,\xi)\tilde{v}(\xi,t)d\xi,\label{err-backstepping2-inverse}
\end{align}
along with the output gains terms as
\begin{align}
p_1(x)=&-\lambda_2(0)P^{\alpha\beta}(x,0),\label{p1}\\
p_2(x)=&-\lambda_2(0)P^{\beta\beta}(x,0).\label{p2}
\end{align} 
to obtain the following target system:
\begin{align}
    &\partial_t \tilde{\alpha}(x,t) =-\lambda_1(x) \partial_x \tilde{\alpha}(x,t),\label{err-targetalpha} \\
    &\partial_t \tilde{\beta}(x,t)=\lambda_2(x) \partial_x \tilde{\beta}(x,t),\label{err-targetbeta}\\
    &\tilde{\alpha}(0,t)  =0, \quad
    \tilde{\beta}(1,t)  =\rho \tilde{\alpha}(1,t).\label{err-targetbc}
\end{align}
The backstepping transformation kernels in \eqref{err-backstepping1}-\eqref{err-backstepping2-inverse} defined over the triangular domain $0\leq\xi\leq x\leq1$ satisfy
\begin{align*}
&\Lambda(x)\partial_xP(x,\xi) + \partial_\xi(P(x,\xi)\Lambda(\xi))=C(x)P(x,\xi),\\
&P^{\alpha\alpha}(1,\xi)=\tfrac{1}{\rho}P^{\beta\alpha}(1,\xi),
~P^{\alpha\beta}(x,x)=-\tfrac{c_1(x)}{\lambda_1(x)+\lambda_2(x)},\\
&P^{\beta\beta}(1,\xi)=\rho P^{\alpha\beta}(1,\xi),
~P^{\beta\alpha}(x,x)=\tfrac{c_2(x)}{\lambda_1(x)+\lambda_2(x)}.\\
&\Lambda(x)\partial_xR(x,\xi) + \partial_\xi(R(x,\xi)\Lambda(\xi))=-R(x,\xi)C(x),\\
&R^{uu}(1,\xi)=\tfrac{1}{\rho}R^{vu},
~R^{uv}(x,x)=-\tfrac{c_1(x)}{\lambda_1(x)+\lambda_2(x)},\\
&R^{vv}(1,\xi)=\rho P^{uv}(1,\xi),
~R^{vu}(x,x)=\tfrac{c_2(x)}{\lambda_1(x)+\lambda_2(x)},
\end{align*}
where,
{\footnotesize
\begin{align*}
    &\Lambda(x) = \begin{bmatrix}\begin{smallmatrix}\lambda_1(x)&0\\0&-\lambda_2(x)\end{smallmatrix}\end{bmatrix},\quad C(x) = \begin{bmatrix}\begin{smallmatrix}0&c_1(x)\\c_2(x)&0\end{smallmatrix}\end{bmatrix},\\
    &P(x,\xi) = \begin{bmatrix}\begin{smallmatrix}P^{\alpha\alpha}(x,\xi)&P^{\alpha\beta}(x,\xi)\\P^{\beta\alpha}(x,\xi)&P^{\beta\beta}(x,\xi)\end{smallmatrix}\end{bmatrix},~R(x,\xi) = \begin{bmatrix}\begin{smallmatrix}R^{uu}(x,\xi)&R^{uv}(x,\xi)\\R^{vu}(x,\xi)&R^{vv}(x,\xi)\end{smallmatrix}\end{bmatrix}.
\end{align*}}
In contrast to \cite{vazquez2011_2x2hyperbolic}, the kernel equations above are derived under an anti-collocated sensing and actuation framework, leading to a distinct formulation. The control input is determined by sampling a continuous-time signal $U^c(t)$ at times $t_j\in I, j\in\mathbb{N}$, that is we have
\begin{equation}
    U(t) \coloneqq U^c(t_j), t\in[t_j,t_{j+1}), j\in\mathbb{N}.\label{U}
\end{equation}
Let the input holding error be defined as
\begin{align}
d(t) = U(t) - U^c(t).\label{d}
\end{align}
Subsequently, consider the following backstepping transformation
\begin{align}
    \hat{\alpha}(x,t) =&\hat{u}(x,t)-\int_0^xK^{uu}(x,\xi)\hat{u}(\xi,t)d\xi \nonumber\\&-\int_0^xK^{uv}(x,\xi)\hat{v}(\xi,t)d\xi,\label{obs-backstepping1}\\
\hat{\beta}(x,t)  =&\hat{v}(x,t)-\int_0^x K^{vu}(x, \xi) \hat{u}(\xi,t) d\xi\nonumber\\&-\int_0^x K^{vv}(x, \xi) \hat{v}(\xi,t) d\xi,\label{obs-backstepping2} 
\end{align}
and its inverse
\begin{align}
    \hat{u}(x,t)  =&\hat{\alpha}(x,t)+\int_0^x L^{\alpha\alpha}(x, \xi) \hat{\alpha}(\xi,t) d\xi \nonumber\\&+ \int_0^x L^{\alpha\beta}(x,\xi) \hat{\beta}(\xi,t) d\xi, \label{obs-ackstepping1-inverse}\\
    \hat{v}(x,t)  =&\hat{\beta}(x,t)+\int_0^x L^{\beta\alpha}(x, \xi) \hat{\alpha}(\xi,t) d\xi \nonumber\\&+ \int_0^x L^{\beta\beta}(x,\xi) \hat{\beta}(\xi,t) d\xi, \label{obs-backstepping2-inverse}
\end{align}
along with $U^c(t)$ defined as
\begin{equation}\label{cntns-fb}
    U^c(t)=\int_0^1 N^\alpha(\xi) \hat{\alpha}(\xi,t)d\xi+\int_0^1 N^\beta(\xi) \hat{\beta}(\xi,t)d\xi,
\end{equation}
where
\begin{align}
    N^\alpha(\xi)=&L^{\beta\alpha}(1,\xi)-\rho L^{\alpha\alpha}(1,\xi),\label{n-alpha}\\
    N^\beta(\xi)=&L^{\beta\beta}(1,\xi)-\rho L^{\alpha\beta}(1,\xi).\label{n-beta}
\end{align}
Which results in the following target system
\begin{align}
&\partial_t \hat{\alpha}(x,t) =-\lambda_1(x) \partial_x \hat{\alpha}(x,t) + \bar{p}_1(x)\tilde{\beta}(0,t),\label{obs-targetu} \\
&\partial_t \hat{\beta}(x,t)=\lambda_2(x) \partial_x \hat{\beta}(x,t) + \bar{p}_2(x)\tilde{\beta}(0,t),\label{obs-targetv}\\   
&\hat{\alpha}(0,t)=q\hat{\beta}(0,t) + q\tilde{\beta}(0,t),\quad
\hat{\beta}(1,t)=\rho\hat{\alpha}(1,t) + d(t),\label{obs-targetbc}
\end{align}
where
\begin{align}
    \bar{p}_1(x)=&p_1(x)-\int_0^xK^{uu}(x,\xi)p_1(\xi)d\xi\nonumber\\&-\int_0^xK^{uv}(x,\xi)p_2(\xi)d\xi,\label{p1_bar}\\
    \bar{p}_2(x)=&p_2(x)-\int_0^xK^{vu}(x,\xi)p_1(\xi)d\xi\nonumber\\&-\int_0^xK^{vv}(x,\xi)p_2(\xi)d\xi.\label{p2_bar}
\end{align}
We refer the reader to \cite{vazquez2011_2x2hyperbolic} for the definition of the backstepping transformation kernels in \eqref{obs-backstepping1}-\eqref{obs-backstepping2-inverse} over the triangular domain $0\leq\xi\leq x\leq1$. The following proposition addresses the well-posedness issues under the proposed control input.
\begin{proposition}\label{propo:wellposedness}
    Let $j\in\mathbb{N}$, and $U(t)\in\mathbb{R}$ be constant for $t\in[t_j,t_{j+1})$. For a given $(u[t_j], v[t_j])^T \in L^2((0,1);\mathbb{R}^2)$ and $(\hat u[t_j], \hat v[t_j])^T \in L^2((0,1);\mathbb{R}^2)$, there exist unique solutions such that $(u, v)^T \in \mathcal{C}^0([t_j, t_{j+1}];  L^2((0,1);\mathbb{R}^2))$ and $(\hat 
    {u}, \hat{v})^T \in \mathcal{C}^0([t_j,t_{j+1}]; L^2((0,1);\mathbb{R}^2))$ to the systems \eqref{sysu}-\eqref{sysbc} and \eqref{obs-sysu}-\eqref{obs-sysbc} respectively. 
\end{proposition}
\begin{proof}
    Consider the systems \eqref{sysu}-\eqref{sysbc}, and \eqref{obs-sysu}-\eqref{obs-sysbc} for time $t\in[t_j,t_{j+1}]$ for a fixed $j\in\mathbb{N}$ with $U(t)\in\mathbb{R}$ constant for $t\in[t_j,t_{j+1})$ and $(u[t_j], v[t_j])^T \in L^2((0,1);\mathbb{R}^2)$, $(\hat u[t_j], \hat v[t_j])^T \in L^2((0,1);\mathbb{R}^2)$. Then, considering the solution along the characteristics, and following \cite[Appendix A]{coron2021_time_varying_coefficients}, it can be shown that there exists unique solutions $(u, v)^T \in \mathcal{C}^0([t_j, t_{j+1}];  L^2((0,1);\mathbb{R}^2))$ and $(\hat 
    {u}, \hat{v})^T \in \mathcal{C}^0([t_j,t_{j+1}]; L^2((0,1);\mathbb{R}^2))$ to the systems \eqref{sysu}-\eqref{sysbc} and \eqref{obs-sysu}-\eqref{obs-sysbc} respectively.
\end{proof}

The following lemma holds for the input holding error $d(t)$.
\begin{lemma}
    The input holding error $d(t)$ defined in \eqref{d} satisfy the following inequality:
    \begin{equation}
    \begin{aligned}
        \dot d(t) ^2 \leq& \varepsilon_0 d(t) + \varepsilon_1 \|\hat \alpha[t]\|^2 + \varepsilon_2 \|\hat \beta[t]\|^2\\&  + \varepsilon_3 \hat\alpha(1,t)^2+ \varepsilon_4 \tilde \beta(0,t)^2,\label{dotd2}
        \end{aligned}
    \end{equation}
    where
    \begin{align}
        \varepsilon_0 =& 5 (\lambda_2(1) N^\beta(1))^2,\\
        \varepsilon_1 =& 5 \int_0^1 (\partial_\xi\lambda_1(\xi)N^\alpha(\xi)+\lambda_1(\xi)\partial_\xi N^\alpha(\xi))^2d\xi,\\
        \varepsilon_2 =& 5 \int_0^1 (\partial_\xi\lambda_2(\xi)N^\beta(\xi)+\lambda_2(\xi)\partial_\xi N^\beta(\xi))^2d\xi,\\
        \varepsilon_3 =& 5 (\lambda_1(1) N^\alpha(1) - \rho\lambda_2(1)N^\beta(1))^2,\\
        \varepsilon_4 =& 5 \Big(\int_0^1( N^\alpha(\xi)\bar p_1(\xi) + N^\beta(\xi)\bar p_2(\xi))d\xi\nonumber\\
        &\mkern200mu+\lambda_1(0)qN^\alpha(0)\Big)^2.
    \end{align}
\end{lemma}
The proof is similar to \cite[Lemma 2]{espitia2020observer} and hence omitted.

\section{Performance-Barrier Event-Triggered Control}\label{sec:p_etc}
In this section, we present the novel observer-based P-ETC strategy and the stability of the closed-loop system. A performance residual term based only on observer states is incorporated with the ETC mechanism in \cite{bhathiya2026_etc_stability} to achieve increased dwell-times for a control input applied in ZOH between events. Therefore, unlike \cite{ong2024performance,rathnayake2025performance,zhang2025performance}, full state information is not required for P-ETC design presented here. The proposed P-ETC mechanism is illustrated in Fig.~\ref{fig:triggering-mechanism}.
\begin{figure}[t]
    \centering
    \begin{tikzpicture}
    \def\concolor{BrickRed}
    \def\etcolor{Purple}
    \def\plantcolor{RoyalBlue}
    \def\obscolor{Aquamarine}
    \def\ucolor{YellowOrange}

    \node (plant) [rectangle, draw = \plantcolor, text = \plantcolor, minimum width = 2cm, minimum height = 0.5 cm] {Plant};
    \node (y0) [left = 75pt of plant.west] {};
    \node (y1) [below = 20pt of y0] {};
    \node (y2) [below = 10pt of y1] {};
    \node (U0) [right = 75pt of plant.east] {};
    \node (U1) [below = 20pt of U0] {};
    \node (U2) [below = 10pt of U1] {};
    \node (observer) [rectangle, draw = \obscolor, text = \obscolor, minimum width = 1cm, minimum height = 0.5 cm] at ($ (y2)!0.15!(U2) $) {Observer};
    \node (observer2) [right= 32pt of observer.east] {};
    \draw[\obscolor, line width =1pt] (observer.east) to node[pos=0.5, above]{$\hat{u}[t],\hat{v}[t]$} (observer2.center);
    \node (controller) [rectangle, draw = \concolor, text = \concolor, minimum width = 2cm, minimum height = 0.5 cm] at ($ (y1)!0.63!(U1) $) {Controller};
    \node (petc) [rectangle, draw = \etcolor, text = \etcolor, minimum width = 2cm, text width = 2cm, text centered, minimum height = 0.5 cm, below=20pt of controller] {Triggering mechanism};

    \node (s1c) [circle, draw, fill=black, minimum size = 3pt, inner sep=0pt, right = 30pt of controller.east] {};
    \node (s1u) [circle, draw, fill=black, minimum size = 3pt, inner sep=0pt, right = 10pt of s1c.center] {};
    \node (s1etc) [below right = 0pt and 0pt of s1c] {};
    \draw[-,line width = 1.5pt] (s1c.center) to ++(315:10pt);

    \draw[\concolor, -stealth, line width = 1pt] (controller.south) to node[pos=0.5, left]{$U^c(t)$} (petc.north);
    \draw[\plantcolor, -stealth, line width = 1pt] (plant.west) to node[pos=0.15, above]{$y(t)$} (y0.center) to (y2.center) to (observer.west);
    \draw[\obscolor, -stealth, line width =1pt] (observer2.center) to (observer2|-petc.west) to (petc.west);
    \draw[\obscolor, -stealth, line width =1pt] (observer2.center) to (observer2|-controller.west) to (controller.west);
    \draw[\concolor, line width =1pt] (controller.east) to node[pos=0.5, above]{$U^c(t)$} (s1c.west);
    \draw[\ucolor,-stealth, line width =1pt] (s1u.east) to  (U1.center) to (U0.center) to node[pos=0.8, above]{$U(t)$} (plant.east);
    \draw[\etcolor,dashed,-stealth, line width = 1pt] (petc.east) to node[pos=0.5, below]{$\{t_j\}_{j\in\mathbb{N}}$} (petc.east-|s1etc) to  (s1etc);
\end{tikzpicture}
    \caption{Schematic of the proposed P-ETC mechanism}
    \label{fig:triggering-mechanism}
\end{figure}
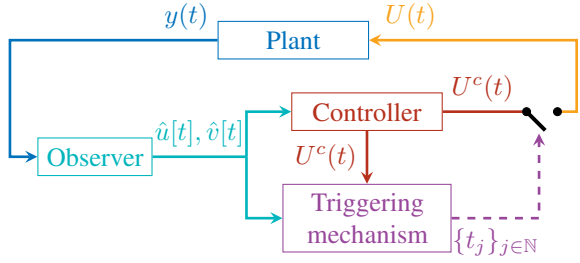
The increasing sequence of times at which events occur is defined according to the definition given below.
\begin{definition}\label{def:etc}
The increasing sequence of times $I$ at which events are generated is defined according to the following rule:
\begin{equation}
    t_{j+1}\coloneqq \inf\{t\geq t_j + \tau\}\vert m(t)<0, j\in\mathbb{N}\}.\label{tj+1}
\end{equation}
    The dynamic variable $m(t)$ is defined as
    \begin{equation}
    \begin{aligned}
        \dot m(t) =& -\eta m(t) + \kappa_1\|\hat \alpha[t]\|^2 + \kappa_2 \|\hat \beta[t]\|^2\\&  + \kappa_3 \hat\alpha(1,t)^2+ \kappa_4 \tilde \beta(0,t)^2 + cW(t),\label{m1}
    \end{aligned}
    \end{equation}
    for $t\in(t_j,t_j+\tau),j\in\mathbb{N}$ and
    \begin{equation}
    \begin{aligned}
        \dot m(t) =& -\eta m(t) -\theta d(t)^2 + \kappa_1 \|\hat \alpha[t]\|^2 + \kappa_2 \|\hat \beta[t]\|^2\\&  + \kappa_3 \hat\alpha(1,t)^2+ \kappa_4 \tilde \beta(0,t)^2 + cW(t),\label{m2}
    \end{aligned}
    \end{equation}
    for $t\in(t_j+\tau,t_{j+1}),j\in\mathbb{N}$. In addition, define $m(0)=0$, $m(t_j)=\omega_0d(t_j^-)^2,j\in\mathbb{N}_{>0}$, and $m(t_j+\tau)=m((t_j+\tau)^-),j\in\mathbb{N}$. The minimum dwell-time $\tau$ is defined such that
    \begin{equation}\label{f}
    \dot f(t) =
        \begin{cases}
            -a_2f(t)^2 - a_1f(t) - a_0&t\in(t_j,t_j+\tau)\\
            0&t\in(t_j+\tau,t_{j+1})
        \end{cases},
    \end{equation}
    with $f(t_j)=\omega_1$, $f((t_j+\tau)^-)=f(t_j+\tau)=\omega_0$\, and 
    \begin{align}
        \tau = \int_{\omega_0}^{\omega_1}\frac{1}{a_2s^2 + a_1s + a_0}ds.\label{tau}
    \end{align}The estimated performance residual, $W(t)$ is defined as
    \begin{equation}
        W(t) = \mathrm{e}^{-\gamma t}\varrho(t) - \hat V(t),\label{w}
    \end{equation}
    where
    \begin{align}
        \varrho(t) =& \max_{s\in[0,t]}\left\{\mathrm{e}^{\gamma s}\hat V(s)\right\},\label{varrho}\\
        V_1(t) =& \int_0^1 \bigg(\frac{A}{\lambda_1(x)}\mathrm{e}^{-\mu\phi_1(x)}\hat \alpha(x,t)^2\notag\\
        &+\frac{B}{\lambda_2(x)}\mathrm{e}^{\mu\phi_2(x)}\hat \beta(x,t)^2\bigg)dx,\label{v1}\\
        \hat V(t) =& V_1(t) + f(t)d(t)^2 + m(t).\label{hat-v}
    \end{align}
    The values $\eta$, $\kappa_1$, $\kappa_2$, $\kappa_3$, $\kappa_4$, $c$, $\omega_1>\omega_0$, and $\gamma\leq \nu$ are positive design parameters. The constants $A$,$B$,$\mu,\theta$, and $\nu$ are to be specified subsequently. 
    \end{definition}
    
    Note that, $\hat V(t)\in \mathcal{C}^0(\mathbb{R}_{\geq0};\mathbb{R}_{\geq0})$ given that $(\hat \alpha,\hat \beta)^T\in \mathcal{C}^0(\mathbb{R}_{\geq0};L^2((0,1);\mathbb{R}^2))$. Hence, $\varrho(t), W(t)\in \mathcal{C}^0(\mathbb{R}_{\geq0};\mathbb{R}_{\geq0})$. In addition, by \eqref{varrho}, $\varrho(t)$ is a non-decreasing function of time $t$. The following lemma holds for the dynamic variable $m(t)$, and the estimated performance residual $W(t)$ in Definition~\ref{def:etc}.
\begin{lemma}
    Consider an increasing sequence of event times $I=\{t_j\}_{j\in\mathbb{N}}$, $t_0=0$, satisfying $\lim_{j\rightarrow\infty}t_j=+\infty$ determined by the triggering rule \eqref{tj+1}-\eqref{hat-v}. Then the estimated performance residual $W(t)$, satisfies $W(t)\geq0 \,\forall t\geq0$. Subsequently, the dynamic variable $m(t)$ governed by \eqref{m1},\eqref{m2} with $m(0)=0$, $m(t_j)\geq 0,j\in\mathbb{N}_{>0}$, $m((t_j+\tau)^-)=m(t_j+\tau),j\in\mathbb{N}$, where $\tau$ is the MDT, it holds that $m(t)\geq0~\forall t\geq0$.
\end{lemma}
\begin{proof}
    Under Definition~\ref{def:etc}, $I$ is an increasing sequence due to the existence of a MDT $\tau$. Therefore, from \eqref{varrho} and \eqref{w}, we have
    \begin{align*}
        W(t) = \mathrm{e}^{-\gamma t}\max_{s\in[0,t]}\left\{\mathrm{e}^{\gamma s}\hat V(s)\right\} - \hat V(t)\geq0,~\forall t\geq0.
    \end{align*}
    Let $m(0)=0$, $m(t_j)\geq 0,j\in\mathbb{N}_{>0}$, $m((t_j+\tau)^-)=m(t_j+\tau),j\in\mathbb{N}$, then from \eqref{m1} we have
    \begin{equation}
    \begin{aligned}
        m(t) = \mathrm{e}^{-\eta(t-t_j)}m(t_j) + \int_{t_j}^t \mathrm{e}^{-\eta(t -\xi)}\big(\kappa_1\|\hat\alpha[\xi]\|^2\\
        \kappa_2 \|\hat \beta[\xi]\|^2  + \kappa_3 \hat\alpha(1,\xi)^2+ \kappa_4 \tilde \beta(0,\xi)^2 + cW(t)\big)d\xi,
    \end{aligned}
    \end{equation}
    $\forall t\in(t_j,t_j+\tau], j\in\mathbb{N}$. Hence $m(t)\geq0~\forall t\in[t_j,t_j+\tau], j\in\mathbb{N}$ Subsequently, for $t\in(t_j+\tau,t_{j+1}), j\in\mathbb{N}$, $m(t)$ evolves as in \eqref{m2} with $t_{j+1}$ selected according to \eqref{tj+1} to ensure $m(t)\geq0,\forall t\in(t_j+\tau,t_{j+1}),j\in\mathbb{N}$. Applying the aforementioned reasoning for each $t_j\in I$ starting with $t_0=0$, it can be shown that $m(t)\geq0, \forall t\geq0$.
\end{proof}

Under Definition~\ref{def:etc}, the evolution of the dynamic variable $m(t)$ associated with the triggering condition \eqref{tj+1} is shown in Fig~\ref{fig:etc}. Where an event time $t_j,j\in\mathbb{N}$ occurs when $m(t)<0$, with $m(0)=0$, $m(t_j)=\omega_0d(t_j^-)^2,j\in\mathbb{N}_{>0}$, and $m(t_j+\tau)=m((t_j+\tau)^-),j\in\mathbb{N}$. Compared to the triggering condition in \cite{bhathiya2026_etc_stability}, a positive estimated performance residual term $W(t)$ is introduced into the dynamics of $m(t)$ allowing $m(t)$ to stay positive longer, resulting in longer dwell-times. 
    \begin{figure}[t]
    \centering
    \begin{tikzpicture}
\node (states) [align=left] {$\kappa_1\|\hat\alpha[t]\|^2$\\ $+\kappa_2\|\hat\beta[t]\|^2$\\$+\kappa_3\hat\alpha(1,t)$\\$+\kappa_4\tilde\beta(0,t)$};
\node (add) [circle,draw,minimum size = 3pt,inner sep=0pt,right= 0.1\columnwidth of states] {$+$};
\node (int) [draw,right= 0.2\columnwidth of add] {$\int$};
\node (m) [right= 0.2\columnwidth of int] {$m(t)$};
\node (w) [above= 0.1cm of states,color=red] {$W(t)$};
\node (c) [red, draw, above= 0.3cm of add] {$c$};
\node (eta) [draw,below= 0.7cm of int] {$-\eta$};
\node (add2) [circle,draw,minimum size = 3pt,inner sep=0pt] at (add |- eta) {$+$};
\node (theta) [blue, draw] at (states |- add2) {$-\theta$};
\node (d)[blue,below= 0.5cm of theta] {$d(t)^2$};

    \node (s1) [circle, draw=blue, fill=blue, minimum size = 3pt, inner sep=0pt, right = 10pt of theta.east] {};
    \node (s2) [circle, draw=blue, fill=blue, minimum size = 3pt, inner sep=0pt, right = 10pt of s1.center] {};
    \node (s) [below right = 0pt and 0pt of s1] {};
    \draw[blue,-,line width = 1.5pt] (s1.center) to ++(315:10pt);

\draw[red, -stealth, line width = 1pt] (w.east) to (w.east-|c.north) to (c.north);
\draw[red, -stealth, line width = 1pt] (c.south) to (add.north);
\draw[-stealth, line width = 1pt] (states.east) to (add.west);
\draw[-stealth, line width = 1pt] (add.east) to (int.west);
\draw[-stealth, line width = 1pt] (int.east) to (m.west);
\draw[-stealth, line width = 1pt] (m.south) to (m.south|-eta.east) to (eta.east);
\draw[-stealth, line width = 1pt] (m.south) to (m.south|-eta.east) to (eta.east);
\draw[-stealth, line width = 1pt] (eta.west) to (add2.east);
\draw[blue,-stealth, line width = 1pt] (d.north) to (theta.south);
\draw[blue,-stealth, line width = 1pt] (theta.east) to (s1.west);
\draw[blue,-stealth, line width = 1pt] (s2.east) to (add2.west);
\draw[blue,stealth-, dashed, line width = 1pt] (s.south)  -- node[midway, right] {$t\in(t_j+\tau,t_{j+1}),j\in\mathbb{N}$} ++(0,-0.7cm);
\draw[-stealth, line width = 1pt] (add2.north) to (add.south);
\end{tikzpicture}
    \caption{Schematic of the dynamics of $m(t)$}
    \label{fig:etc}
    \end{figure}
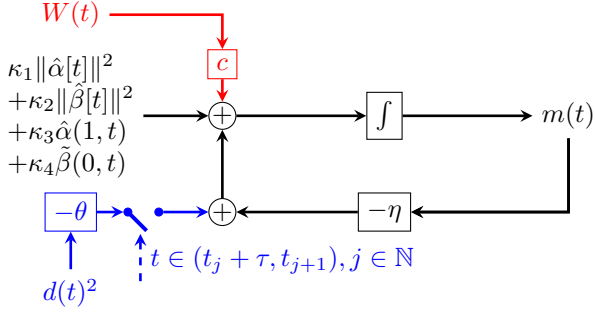

The following stability condition holds under the proposed ETC mechanism.
\begin{theorem}
    Under the event-triggering rule \eqref{tj+1}-\eqref{hat-v} and Assumption~\ref{assum:rhoq}, consider a set of increasing sequence of event times $I=\{t_j\}_{j\in\mathbb{N}}$, $t_0=0$ such that $\lim_{j\rightarrow\infty}t_j=+\infty$. For each $(u^0, v^0, \hat u^0, \hat v^0)^T\in L^2((0,1);\mathbb{R}^4))$, there exists a unique solution $(u, v, \hat u, \hat v)^T\in\mathcal{C}^0(\mathbb{R}_{>0};L^2((0,1);\mathbb{R}^4))$ to \eqref{sysu}-\eqref{sysbc}, \eqref{obs-sysu}-\eqref{obs-sysbc}. Let $\eta$, $\kappa_1$, $\kappa_2$, $\kappa_3$, $\kappa_4$, $a_1$, $a_2$, $c$, and $\omega_1>\omega_0$ be free positive parameters to be selected. In addition, let $a_0$, $\theta$, and $\gamma$ be
\begin{align}
    a_0 =& 4Aq^2\mathrm{e}^{\mu\phi_2(1)} + \frac{\epsilon_0}{a_2},\quad \theta = a_2\omega_0^2 + a_1\omega_0 + a_0,\label{a0-theta}\\
    \gamma\leq&\nu,\quad \nu=\min\left\{\nu_0,a_1,\eta\right\}\label{gamma-nu}\\
    \nu_0 =& \mu-\delta-\tfrac{r}{A}\max\left\{\tfrac{\epsilon_1}{a_2} + \kappa_1,\tfrac{\epsilon_2}{a_2} + \kappa_2\right\}\label{nu0}\\
    \mu\in&\left(0,\frac{2}{\phi_1(1) + \phi_2(1)}\ln\left(\frac{1}{2|\rho q|}\right)\right),~\delta<\mu,\label{mu}\\
    r=&\min\left\{\tfrac{\mathrm{e}^{-\mu\phi_1}}{\phi_1(1)},\tfrac{2q^2}{\phi_2(1)}\right\}^{-1},\label{r}\\
    A >& \max\bigg\{\tfrac{\mathrm{e}^{\mu\phi_1(1)}}{1-4\rho^2q^2\mathrm{e}^{\mu\left(\phi_1(1) + \phi_2(1)\right)}}\left(\tfrac{\epsilon_3}{a_2}+\kappa_3\right)\notag\\
    &\tfrac{r}{\mu-\delta}\max\{\tfrac{\epsilon_1}{a_2}+\kappa_1,\tfrac{\epsilon_2}{a_2}+\kappa_2\}\bigg\}.\label{A}
\end{align}
then the solution of the closed-loop system \eqref{sysu}-\eqref{sysbc}, \eqref{obs-sysu}-\eqref{obs-sysbc} with the control input $U(t)$ given by \eqref{U} updated at each $t_j\in I$, is globally exponentially stable in the spatial $L^2$ norm.
\end{theorem}
\begin{proof}
    By the existence of a MDT $\tau$ imposed by the triggering condition \eqref{tj+1} along with Proposition~\ref{propo:wellposedness}, the existence of a unique solution $(u, v, \hat u, \hat v)^T\in\mathcal{C}^0(\mathbb{R}_{>0};L^2((0,1);\mathbb{R}^4))$ to \eqref{sysu}-\eqref{sysbc}, \eqref{obs-sysu}-\eqref{obs-sysbc} can be shown. Subsequently, consider the following Lyapunov function candidate
    \begin{align}
        V(t) =& \hat V(t) + V_2(t), \label{v}\\
        V_2(t) =& \int_0^1 \bigg(\frac{C}{\lambda_1(x)}\mathrm{e}^{-\mu\phi_1(x)}\tilde \alpha(x,t)^2\\
        &+\frac{D}{\lambda_2(x)}\mathrm{e}^{\mu\phi_2(x)}\tilde \beta(x,t)^2\bigg)dx,\label{v2}
    \end{align}
    where $\hat V(t)$ is defined in \eqref{hat-v}. Taking the derivative of \eqref{v} with respect to time $t\in(t_j,t_j+\tau)\cup(t_j+\tau,t_{j+1}),j\in\mathbb{N}$ along the solution of \eqref{err-targetalpha}-\eqref{err-targetbc},\eqref{obs-targetu}-\eqref{obs-targetbc}, using Young's inequality, and \eqref{dotd2}, we have
    \begin{align*}
        &\dot V(t) \leq -\left(A\mathrm{e}^{-\mu\phi_1(1)} - 2B\rho^2\mathrm{e}^{\mu\phi_2(1)}-\tfrac{\epsilon_3}{a_2} - \kappa_3\right)\hat \alpha(1,t)^2\\
        &-\left(B-2Aq^2\right)\hat \beta(0,t)^2 - \left(C\mathrm{e}^{-\mu\phi_1(1)} - D\rho^2\mathrm{e}^{\mu\phi_2(1)}\right)\tilde \alpha(1,t)^2\\
        &-\left(D - 2Aq^2 - \tfrac{A\|\bar p_1\|^2}{\underline{\lambda}_1\delta} - \tfrac{B\mathrm{e}^{\mu\phi_2(1)}\|\bar p_2\|^2}{\underline{\lambda}_2\delta} - \tfrac{\epsilon_4}{a_2} - \kappa_4\right)\tilde \beta(0,t)^2\\
        &-(\mu-\delta)V_1(t) - \mu V_2(t)\\
        &+ \max\left\{\tfrac{\epsilon_1}{a_2} + \kappa_1,\tfrac{\epsilon_2}{a_2} + \kappa_2\right\}(\|\hat \alpha[t]\|^2 + \|\hat \beta[t]\|^2)\\
        &+\left(\dot f(t) + a_2f(t)^2+ \tfrac{\epsilon_0}{a_2} + 2B\mathrm{e}^{-\mu\phi_2(1)}\right)d(t)^2 + \dot m(t)\\
        & - \kappa_1\|\hat \alpha[t]\|^2 - \kappa_2\|\hat \beta[t]\|^2 - \kappa_3\hat \alpha(1,t)^2 - \kappa_4\tilde \beta(0,t)^2,
    \end{align*}
    for some $\delta>0$ and $a_2>0$. Selecting $B=Aq^2$, $C=D\rho^2\mathrm{e}^{\mu\left(\phi_1(1)+\phi2(1)\right)}$, $\nu_0$ as in \eqref{nu0}, and $r$ as in \eqref{r}, the following expression holds:
    \begin{align*}
        &\dot V(t) \leq-\nu_0V_1(t) - \mu V_2(t)-\left(D - 2Aq^2 - \tfrac{A\|\bar p_1\|^2}{\underline{\lambda}_1\delta}\right.\\
        &\left.- \tfrac{B\mathrm{e}^{\mu\phi_2(1)}\|\bar p_2\|^2}{\underline{\lambda}_2\delta}- \tfrac{\epsilon_4}{a_2} - \kappa_4\right)\tilde \beta(0,t)^2\\
        &-\left(A\left(\mathrm{e}^{-\mu\phi_1(1)}-4\rho^2q^2\mathrm{e}^{\mu\phi_2(1)}\right)-\tfrac{\epsilon_3}{a_2}- \kappa_3\right)\hat \alpha(1,t)^2\\
        &+\left(\dot f(t) + a_2f(t)^2+ \tfrac{\epsilon_0}{a_2} + 2B\mathrm{e}^{-\mu\phi_2(1)}\right)d(t)^2\\
        & + \dot m(t)- \kappa_1\|\hat \alpha[t]\|^2 - \kappa_2\|\hat \beta[t]\|^2 - \kappa_3\hat \alpha(1,t)^2 - \kappa_4\tilde \beta(0,t)^2. 
    \end{align*}
    Under Assumption~\ref{assum:rhoq}, let $\mu$ satisfy \eqref{mu}. Subsequently, let $\delta$ be such that $\delta<\mu$, $A$ as in \eqref{A}, and $D$ as
    \begin{align*}
        D = 2Aq^2 + \tfrac{A\|\bar p_1\|^2}{\underline{\lambda}_1\delta} + \tfrac{B\mathrm{e}^{\mu\phi_2(1)}\|\bar p_2\|^2}{\underline{\lambda}_2\delta} + \tfrac{\epsilon_4}{a_2} + \kappa_4.
    \end{align*}
    Hence, we obtain the following expression for $t\in(t_j,t_j+\tau)\cup(t_j+\tau,t_{j+1}),j\in\mathbb{N}$:
    \begin{align}
        &\dot V(t) \leq-\nu_0V_1(t) - \mu V_2(t)\notag\\
        &+\left(\dot f(t) + a_2f(t)^2+ \tfrac{\epsilon_0}{a_2} + 2B\mathrm{e}^{-\mu\phi_2(1)}\right)d(t)^2+ \dot m(t)\notag\\
        & - \kappa_1\|\hat \alpha[t]\|^2 - \kappa_2\|\hat \beta[t]\|^2 - \kappa_3\hat \alpha(1,t)^2 - \kappa_4\tilde \beta(0,t)^2.\label{dotv-1}
    \end{align}
    Consider \eqref{dotv-1} for $t\in(t_j,t_j+\tau),j\in\mathbb{N}$. Using \eqref{m1} and \eqref{f} with $a_0$, $\nu$ as in \eqref{a0-theta},\eqref{gamma-nu} respectively, the following expression holds:
    \begin{align}
        \dot V(t)\leq -\nu V(t) + cW(t).\label{dotv-2}
    \end{align}
    Similarly, consider \eqref{dotv-1} for $t\in(t_j+\tau,t_{j+1}),j\in\mathbb{N}$. Using \eqref{m2} and \eqref{f} with $\theta$, $\nu$ as in \eqref{a0-theta},\eqref{gamma-nu} respectively, the following expression holds:
    \begin{align}
        \dot V(t)\leq -\nu V(t) + cW(t).\label{dotv-3}
    \end{align}
    Let $\gamma\leq\nu$ as in \eqref{gamma-nu}. To simplify the subsequent derivations, define
    \begin{align}
        V_3(t) = \mathrm{e}^{\gamma t}\hat V(t),\quad V_4(t) = \mathrm{e}^{\gamma t} V_2(t).\label{v3-v4}
    \end{align}
    Consequently, from \eqref{w},\eqref{dotv-2},\eqref{dotv-3}, and \eqref{v3-v4} the following expression holds for $t\in(t_j,t_j+\tau)\cup(t_j+\tau,t_{j+1}),j\in\mathbb{N}$:
    \begin{align}
        \dot V_3(t) + \dot V_4(t) \leq c\varrho(t) - cV_3(t).\label{dotv-4}
    \end{align}
    Note that we have $(\hat \alpha,\hat \beta)^T\in \mathcal{C}^0(\mathbb{R}_{\geq0};L^2((0,1);\mathbb{R}^2))$ and $(\tilde \alpha,\tilde \beta)^T\in \mathcal{C}^0(\mathbb{R}_{\geq0};L^2((0,1);\mathbb{R}^2))$. Therefore, $V_3,V_4,\varrho\in \mathcal{C}^0(\mathbb{R}_{\geq0};\mathbb{R}_{\geq0})$.
    
    Taking the derivative of \eqref{v2} with respect to time $t\geq0$, it can be shown that 
    \begin{align}
        \dot V_4(t) \leq& 0,\notag\\
        V_4(t) \leq& V_4(0).\label{v4-bound}
    \end{align}
    Now consider times $t^*\in(t_j,t_j+\tau)\cup(t_j+\tau,t_{j+1}),j\in\mathbb{N}$ at which $\varrho(t^*)$ increases. Note that by definition of $\varrho(t)$ in \eqref{varrho}, $\varrho(t^*) = V_3(t^*)$. From \eqref{dotv-4}
    \begin{align}
        \dot V_3(t^*) \leq -\dot V_4(t^*), \label{dotv-6}
    \end{align}
    since \eqref{dotv-6} is holds when $\varrho(t)$ is updated and $\varrho(t)$ is constant otherwise, the following exists for $t\in(t_j,t_j+\tau)\cup(t_j+\tau,t_{j+1}),j\in\mathbb{N}$:
    \begin{align}
        D^+ \varrho(t) \leq -\dot V_4(t),
    \end{align}
    where we denote $D^+$ as the upper right Dini derivative. Using \cite[Lemma 3.4]{khalil2002nonlinear} for each interval in $(t_j,t_j+\tau)\cup(t_j+\tau,t_{j+1}),j\in\mathbb{N}$, the following holds:
    \begin{align*}
        \varrho(t) - \varrho(0) \leq& V_4(0) - V_4(t),\\
        \varrho(t)\leq& V_3(0) + V_4(0) -V_4(t),\\
        \varrho(t)\leq& V_3(0) + V_4(0).
    \end{align*}
    Hence we have
    \begin{align}
        V_3(t) \leq V_3(0) + V_4(0), \forall t\geq0.\label{v3-bound}
    \end{align}
    Using \eqref{v3-v4} and \eqref{v4-bound}, \eqref{v3-bound}, we have
    \begin{align}
        &\hat V(t) + V_2(t) \leq \mathrm{e}^{-\gamma t}(\hat V(0) + 2V_2(0)),\label{finalv-bound}
    \end{align}
    Finally, using \eqref{finalv-bound} and the invertibility of the backstepping transformations, it can be shown that the solution of the closed-loop system \eqref{sysu}-\eqref{sysbc}, \eqref{obs-sysu}-\eqref{obs-sysbc} with the control input $U(t)$ given by \eqref{U} updated at each $t_j\in I$, is globally exponentially stable in the spatial $L^2$ norm.
\end{proof}
\section{Numerical Simulations}\label{sec:simulations}
We consider the system \eqref{sysu}-\eqref{sysbc} with plant parameters $\lambda_1(x)=1$, $\lambda_2(x)=1$, $c_1(x)=1$ and $c_2(x)=1.5$ $\forall x\in[0,1]$, and $q=0.5$, $\rho=0.5$ to satisfy Assumption~\ref{assum:rhoq}. The initial conditions are selected as $u^0(x)=qv^0(x)$, $v^0(x)=10(1-x)$ and $\hat u^0(x)=0$, $\hat v^0(x)=0$, $\forall x\in[0,1]$. 

The free triggering parameters are selected as $\eta=1$, $a_1=1$, $a_2=1$, $\kappa_1=1$, $\kappa_2=1$, $\kappa_3=1$, $\kappa_4=1$, $\omega_0=1$, $\omega_1=10$, and $c=1$. We select $\mu=0.173 $ such that \eqref{mu} is satisfied and $\delta=0.17$ such that $\delta<\mu$. Subsequently, $A=253.75$ is defined such that \eqref{A} holds. Finally, we select $\gamma = 0.1236$ such that $\gamma\leq\nu$. Under the proposed system and the parameters selected, the MDT is $\tau=0.0137s$.

Fig.~\ref{fig:norm} shows the evolution of the spatial $L^2$ norm of the states over time under the P-ETC input \eqref{U} presented in Fig.~\ref{fig:U}. The resulting dwell-times are shown in Fig.~\ref{fig:dwell_times}. Figures~\ref{fig:norm},\ref{fig:U},\ref{fig:dwell_times} also show the evolution of the system under ETC, that is, without the effect of the performance residual (with $c=0$). The evolution of $\hat{V}(t)$ in \eqref{hat-v} together with the performance barrier $\mathrm{e}^{-\gamma t}\varrho(t)$ is illustrated in Fig.~\ref{fig:V}. The design parameter $c$ in \eqref{m1},\eqref{m2} provides a means to increase the dwell-times at the expense of the convergence rate of $\hat{V}(t)$ to zero. Fig.~\ref{fig:c} compares the evolution of $\hat{V}(t)$ with time for various values of $c$.
\begin{figure}[t]
\centering
\includegraphics[width=\columnwidth]{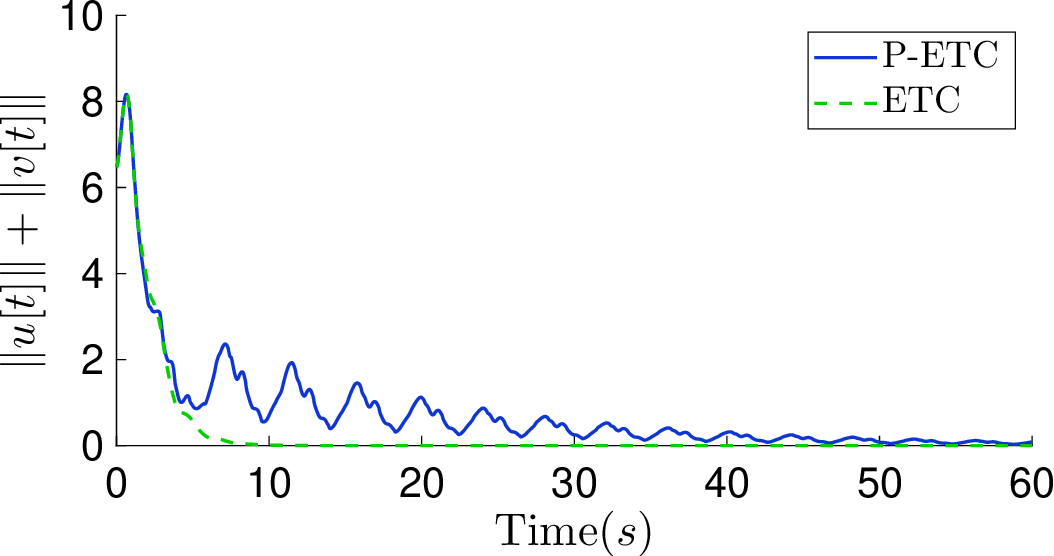}
\caption{Variation of the spatial $L^2$ norm of the states with time}\label{fig:norm}
\end{figure}
\begin{figure}[t]
\centering
\includegraphics[width=\columnwidth]{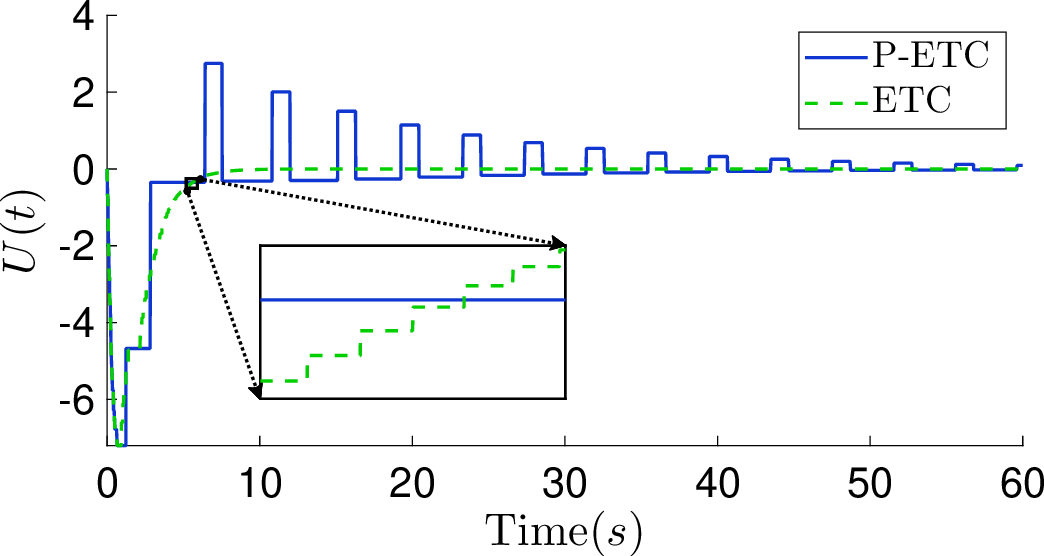}
\caption{Variation of the control input with time}\label{fig:U}
\end{figure}
\begin{figure}[t]
\centering
\includegraphics[width=\columnwidth]{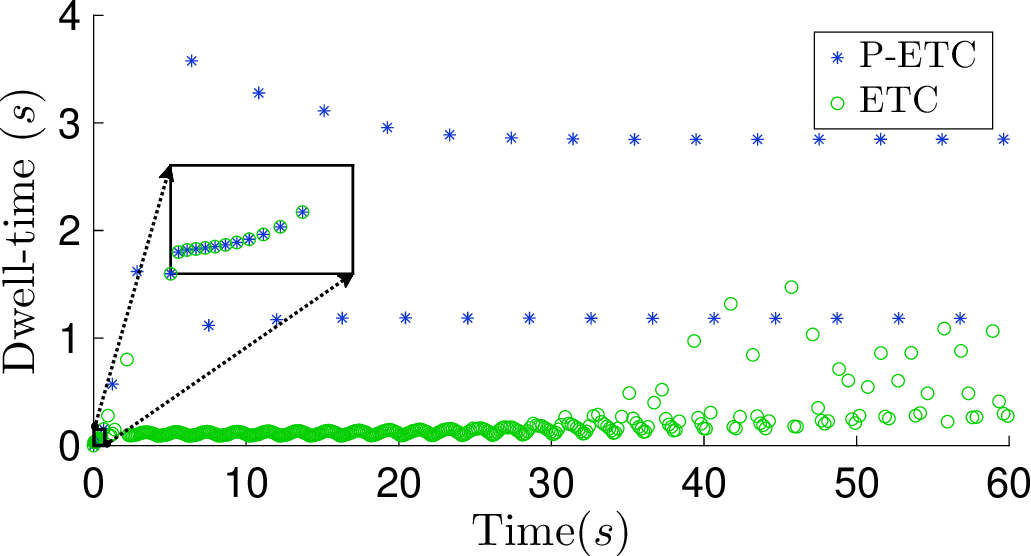}
\caption{Dwell-times under the proposed triggering condition}\label{fig:dwell_times}
\end{figure}
\begin{figure}[t]
\centering
\includegraphics[width=\columnwidth]{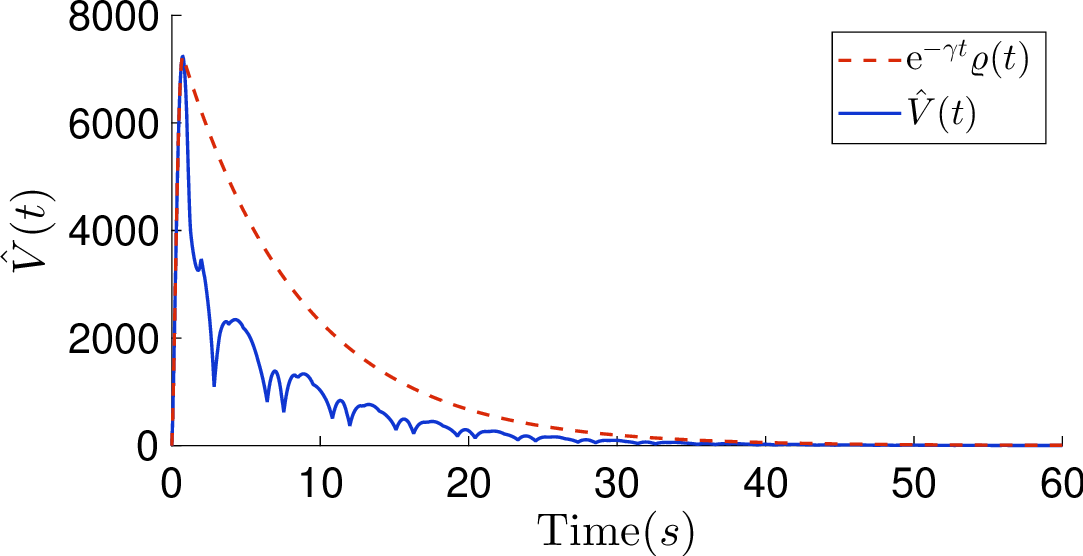}
\caption{Variation of $\hat V(t)$ with time}\label{fig:V}
\end{figure}
\begin{figure}[t]
\centering
\includegraphics[width=\columnwidth]{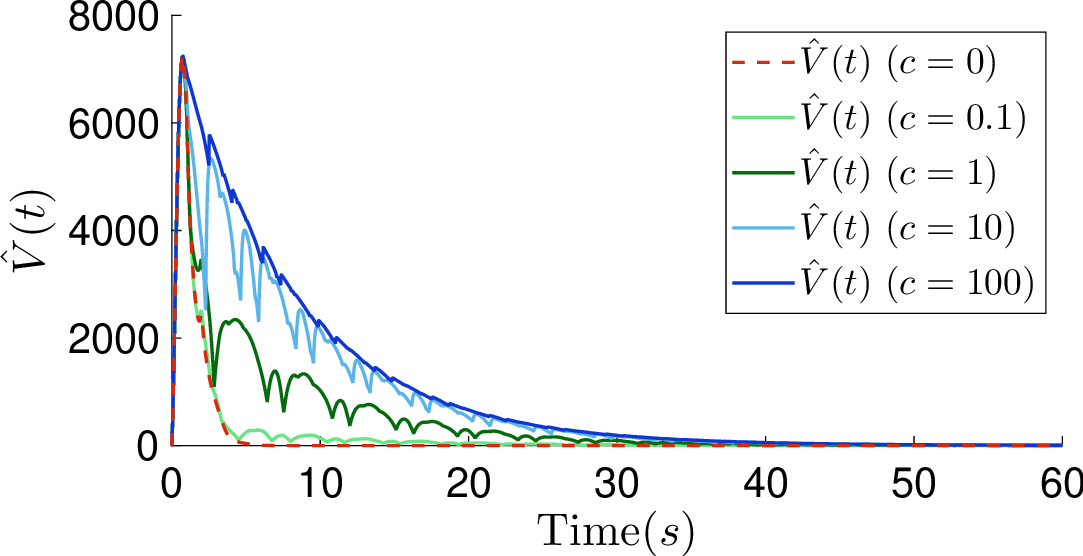}
\caption{Variation of $\hat V(t)$ with time for different values of $c$}\label{fig:c}
\end{figure}

It can be observed that the dwell-times under P-ETC are significantly shorter and coincidences with the dwell-times under ETC initially compared to the later stages. This behavior can be explained by examining Fig.~\ref{fig:V} together with the performance residual $W(t)$ in \eqref{w} and the performance-barrier $\mathrm{e}^{-\gamma t}\varrho(t)$ for $\varrho(t)$ in \eqref{varrho}. Due to the selected initial values of the observer states,  $\varrho(t)$ initially increases. During this phase, the performance residual satisfies $W(t)=0$, causing the system to behave similarly to the standard observer-based ETC scheme. Once $\varrho(t)$ is sufficiently large and its growth ceases,  the performance residual becomes non-zero and influences the evolution of the dynamic triggering variable $m(t)$, producing longer dwell-times.
\section{Conclusion}\label{sec:conclusion}
We present a novel observer-based P-ETC framework for a class of $2\times2$ linear hyperbolic PDEs. The proposed triggering condition guarantees GES of the spatial $L^2$ norm of the solution of the system while substantially increasing the dwell-times between consecutive events compared to ETC that do not incorporate a performance barrier. Simulation results are provided to support the theoretical findings.  The proposed P-ETC framework can be extended to existing dynamic observer-based ETC approaches, provided that the observer converges exponentially to the system states and operates independently of the triggering condition. Investigating such extensions constitutes an interesting direction for future research.

\bibliographystyle{IEEEtranS}
\bibliography{references}

@article{tabuada2007event,
  title={Event-triggered real-time scheduling of stabilizing control tasks},
  author={Tabuada, Paulo},
  journal={IEEE Transactions on Automatic control},
  volume={52},
  number={9},
  pages={1680--1685},
  year={2007},
  publisher={IEEE}
}

@inproceedings{heemels2012introduction,
  title={An introduction to event-triggered and self-triggered control},
  author={Heemels, Wilhelmus PMH and Johansson, Karl Henrik and Tabuada, Paulo},
  booktitle={2012 IEEE Conference on Decision and Control (CDC)},
  pages={3270--3285},
  year={2012}
}

@ARTICLE{girard2015dynamic,
  author={Girard, Antoine},
  journal={IEEE Transactions on Automatic Control}, 
  title={Dynamic Triggering Mechanisms for Event-Triggered Control}, 
  year={2015},
  volume={60},
  number={7},
  pages={1992-1997}
}

@article{ong2024performance,
  title={Performance-barrier-based event-triggered control with applications to network systems},
  author={Ong, Pio and Cort{\'e}s, Jorge},
  journal={IEEE Transactions on Automatic Control},
  year={2024},
  volume={69},
  number={7},
  pages={4230-4244},
  publisher={IEEE}
}

@article{bhathiya2026_etc_stability,
title = {Global exponential stabilization of 2 × 2 linear hyperbolic {PDE}s via dynamic event-triggered backstepping control},
journal = {Automatica},
volume = {183},
pages = {112617},
year = {2026},
issn = {0005-1098},
author = {Bhathiya Rathnayake and Mamadou Diagne}
}

@article{rathnayake2021observer,
  title={Observer-based event-triggered boundary control of a class of reaction-diffusion {PDEs}},
  author={Rathnayake, Bhathiya and Diagne, Mamadou and Espitia, Nicol{\'a}s and Karafyllis, Iasson},
  journal={IEEE Transactions on Automatic Control},
  volume={67},
  number={6},
  pages={2905--2917},
  year={2022},
  publisher={IEEE}
}

@article{espitia2020observer,
  title={Observer-based event-triggered boundary control of a linear $2 \times 2$  hyperbolic systems},
  author={Espitia, Nicol{\'a}s},
  journal={Systems \& Control Letters},
  volume={138},
  pages={104668},
  year={2020},
  publisher={Elsevier}
}

@article{rathnayake2025performance,
  title={Performance-barrier event-triggered control of a class of reaction--diffusion {PDEs}},
  author={Rathnayake, Bhathiya and Diagne, Mamadou and Cort{\'e}s, Jorge and Krstic, Miroslav},
  journal={Automatica},
  volume={174},
  pages={112181},
  year={2025},
  publisher={Elsevier}
}

@article{zhang2025performance,
  title={Performance-barrier event-triggered {PDE} control of traffic flow},
  author={Zhang, Peihan and Rathnayake, Bhathiya and Diagne, Mamadou and Krstic, Miroslav},
  journal={IEEE Transactions on Automatic Control},
  year={2025},
  publisher={IEEE}
}

@article{vazquez2026_backstepping_survery,
title = {Backstepping for partial differential equations: A survey},
journal = {Automatica},
volume = {183},
pages = {112572},
year = {2026},
author = {Rafael Vazquez and Jean Auriol and Federico Bribiesca-Argomedo and Miroslav Krstic}
}

@INPROCEEDINGS{vazquez2011_2x2hyperbolic,
  author={Vazquez, Rafael and Krstic, Miroslav and Coron, Jean-Michel},
  booktitle={2011 50th IEEE Conference on Decision and Control and European Control Conference}, 
  title={Backstepping boundary stabilization and state estimation of a $2\times2$ linear hyperbolic system}, 
  year={2011},
  volume={},
  number={},
  pages={4937-4942}
}

@article{krstic2008_1xhyperbolic,
title = {Backstepping boundary control for first-order hyperbolic {PDEs} and application to systems with actuator and sensor delays},
journal = {Systems \& Control Letters},
volume = {57},
number = {9},
pages = {750-758},
year = {2008},
issn = {0167-6911},
author = {Miroslav Krstic and Andrey Smyshlyaev}
}

@ARTICLE{meglio2013,
  author={Di Meglio, Florent and Vazquez, Rafael and Krstic, Miroslav},
  journal={IEEE Transactions on Automatic Control}, 
  title={Stabilization of a System of  $n+1$ Coupled First-Order Hyperbolic Linear {PDEs} With a Single Boundary Input}, 
  year={2013},
  volume={58},
  number={12},
  pages={3097-3111}}

@ARTICLE{hu2016,
  author={Hu, Long and Di Meglio, Florent and Vazquez, Rafael and Krstic, Miroslav},
  journal={IEEE Transactions on Automatic Control}, 
  title={Control of Homodirectional and General Heterodirectional Linear Coupled Hyperbolic {PDEs}}, 
  year={2016},
  volume={61},
  number={11},
  pages={3301-3314}}

@book{anfinsen2019adaptive,
  title={Adaptive control of hyperbolic {PDE}s},
  author={Anfinsen, Henrik and Aamo, Ole Morten},
  year={2019},
  publisher={Springer}
}

@ARTICLE{coron2007,
  author={Coron, Jean-Michel and d'Andrea-Novel, Brigitte and Bastin, Georges},
  journal={IEEE Transactions on Automatic Control}, 
  title={A Strict {Lyapunov} Function for Boundary Control of Hyperbolic Systems of Conservation Laws}, 
  year={2007},
  volume={52},
  number={1},
  pages={2-11}
  }

@article{litrico2009,
title = {Boundary control of hyperbolic conservation laws using a frequency domain approach},
journal = {Automatica},
volume = {45},
number = {3},
pages = {647-656},
year = {2009},
issn = {0005-1098},
author = {Xavier Litrico and Vincent Fromion}
}

@article{diagne2017backstepping,
  title={Backstepping stabilization of the linearized {Saint-Venant-Exner} model},
  author={Diagne, Ababacar and Diagne, Mamadou and Tang, Shuxia and Krstic, Miroslav},
  journal={Automatica},
  volume={76},
  pages={345--354},
  year={2017}
}

@article{somathilake2024_sediment,
title = {Output Feedback Control of Suspended Sediment Load Entrainment in Water Canals and Reservoirs},
journal = {IFAC-PapersOnLine},
volume = {58},
number = {28},
pages = {983-988},
year = {2024},
note = {The 4th Modeling, Estimation, and Control Conference – 2024},
issn = {2405-8963},
author = {Eranda Somathilake and Mamadou Diagne}
}

@article{somathilake2025_etc,
title = {Output feedback periodic event-triggered and self-triggered boundary control of coupled $2\times2$ linear hyperbolic {PDEs}},
journal = {Automatica},
volume = {179},
pages = {112433},
year = {2025},
issn = {0005-1098},
author = {Eranda Somathilake and Bhathiya Rathnayake and Mamadou Diagne}
}

@article{yu2019traffic,
  title={Traffic congestion control for {Aw--Rascle--Zhang} model},
  author={Yu, Huan and Krstic, Miroslav},
  journal={Automatica},
  volume={100},
  pages={38--51},
  year={2019},
  publisher={Elsevier}
}

@INPROCEEDINGS{bhathiya2025_gas,
  author={Rathnayake, Bhathiya and Zlotnik, Anatoly and Tokareva, Svetlana and Diagne, Mamadou},
  booktitle={2025 American Control Conference (ACC)}, 
  title={Setpoint Tracking and Disturbance Attenuation for Gas Pipeline Flow Subject to Uncertainties using Backstepping}, 
  year={2025},
  volume={},
  number={},
  pages={996-1001}
}

@article{coron2021_time_varying_coefficients,
title = {Boundary stabilization in finite time of one-dimensional linear hyperbolic balance laws with coefficients depending on time and space},
journal = {Journal of Differential Equations},
volume = {271},
pages = {1109-1170},
year = {2021},
issn = {0022-0396},
author = {Jean-Michel Coron and Long Hu and Guillaume Olive and Peipei Shang}
}

@book{khalil2002nonlinear,
  title={Nonlinear systems},
  author={Khalil, Hassan K and Grizzle, Jessy W},
  volume={3},
  year={2002},
  publisher={Prentice hall Upper Saddle River, NJ}
}
\end{document}